\documentclass[preprint,12pt]{elsarticle}
\usepackage{amsfonts}
\usepackage{latexsym}
\usepackage{amsmath}
\usepackage{amssymb}
\usepackage{amscd}
\usepackage{epsfig}
\addtocounter{MaxMatrixCols}{4}
\newtheorem{definition}{Definition}[section]
\newtheorem{proposition}[definition]{Proposition}

\newtheorem{theorem}[definition]{Theorem}

\newtheorem{conjecture}[definition]{Conjecture}
\newenvironment{proof}{\par {\sc {\bf Proof.}\hskip 5pt}}{\hfill \qed \par}
\newcommand{\undertilde}[1]{\ensuremath{\mathord{\vtop{\ialign{##\crcr
   $\hfil\displaystyle{#1}\hfil$\crcr\noalign{\kern1.5pt\nointerlineskip}
   $\hfil\tilde{}\hfil$\crcr\noalign{\kern1.5pt}}}}}}

\journal{Computers \& Operations Research}

\begin{document}

\begin{frontmatter}

\title{A New Heuristic for Detecting Non-Hamiltonicity in Cubic Graphs}
\author[jf]{Jerzy A Filar}
\ead{jerzy.filar@flinders.edu.au}
\author[jf]{Michael Haythorpe\corref{cor1}}
\ead{michael.haythorpe@flinders.edu.au}
\author[jf]{Serguei Rossomakhine}
\ead{serguei.rossomakhine@flinders.edu.au}
\cortext[cor1]{Corresponding author: Michael Haythorpe. Ph: +61 8 820 12375.}
\address[jf]{Flinders University, Sturt Road, Bedford Park SA 5042, Australia.}

\begin{abstract}We analyse a polyhedron which contains the convex hull of all Hamiltonian cycles of a given undirected connected cubic graph. Our constructed polyhedron is defined by polynomially-many linear constraints in polynomially-many continuous (relaxed) variables. Clearly, the emptiness of the constructed polyhedron implies that the graph is non-Hamiltonian. However, whenever a constructed polyhedron is non-empty, the result is inconclusive. Hence, the following natural question arises: if we assume that a non-empty polyhedron implies Hamiltonicity, how frequently is this diagnosis incorrect? We prove that, in the case of bridge graphs, the constructed polyhedron is always empty. We also demonstrate that some non-bridge non-Hamiltonian cubic graphs induce empty polyhedra as well. We compare our approach to the famous Dantzig-Fulkerson-Johnson relaxation of a TSP, and give empirical evidence which suggests that the latter is infeasible if and only if our constructed polyhedron is also empty. By considering special edge cut sets which are present in most cubic graphs, we describe a heuristic approach, built on our constructed polyhedron, for which incorrect diagnoses of non-Hamiltonian graphs as Hamiltonian appear to be very rare. In particular, for cubic graphs containing up to 18 vertices, only four out of 45,982 undirected connected cubic graphs were so misdiagnosed. By constrast, we demonstrate that an equivalent heuristic, when built on the Dantzig-Fulkerson-Johnson relaxation of a TSP, is mostly unsuccessful in identifying additional non-Hamiltonian graphs. These empirical results suggest that polynomial algorithms based on our constructed polyhedron may be able to correctly identify Hamiltonicity of a cubic graph in all but rare cases.\end{abstract}


\end{frontmatter}


\vspace*{-0.5cm}
\section{Introduction}\label{sec-Introduction}
\vspace*{-0.3cm}

The {\em Hamiltonian cycle problem} (HCP) is a well-known problem that features prominently in
complexity theory because it is {\em NP-complete} \cite{gareyjohnson}. The HCP can be stated simply: given
a graph $\Gamma$ containing $N$ vertices, determine whether $\Gamma$ contains a simple cycle of length $N$,
or not. Such a cycle is called a {\em Hamiltonian cycle}. Graphs containing at least one Hamiltonian cycle are called {\em Hamiltonian} graphs, and those containing no Hamiltonian cycles are called {\em non-Hamiltonian} graphs. There are many specialised heuristics which attempt to solve HCP, which include rotational transformation algorithms, cycle extension algorithms, long path algorithms, low degree vertices algorithms, multipath search and pruning algorithms. Attempts to solve HCP have also been made by operations research or optimisation communities, such as nonlinear optimisation (e.g. see Filar et al \cite{walter}) and importance sampling
(e.g. see Eshragh et al \cite{ali}). The HCP is also closely related to the famous Travelling Salesman Problem (TSP), which is simply the problem of finding the Hamiltonian cycle of optimal length. In the language of the TSP, Hamiltonian cycles are usually called {\em tours}.

A cubic graph is one in which every vertex has degree three. The HCP is known to be NP-complete even if only undirected cubic
graphs are considered. By assuming cubicity, there is much inherent graph structure that can be taken advantage of by
algorithms (e.g. see Eppstein \cite{eppstein}). Indeed, there is still a lot of interest in special properties of not only all cubic graphs,
but even special classes of cubic graphs (e.g. see Horev et al \cite{horev}).

More generally, in literature, there have been many approaches towards developing polyhedral sets whose extreme points correspond to solutions of interest. In the case of both TSP and HCP, such a polyhedron is the convex hull of points that are in 1-to-1 correspondence to Hamiltonian cycles (or tours). Let us denote such a polyhedron by $Q := Q(\Gamma)$ for a given graph $\Gamma$. Of course, $Q = \emptyset$ when $\Gamma$ is non-Hamiltonian. In the TSP literature, some of the most successful theories and algorithms have been based on characterisations of facets of $Q$ (e.g. see Gr\"{o}tschel and Padberg \cite{groetschel}). In the context of HCP, however, explicit identification of a Hamiltonian cycle is not necessary. Indeed, Hamiltonicity is equivalent to the determination that $Q \ne \emptyset$. This paper is motivated by the obvious observation that, if a set $\mathcal{P} \supset Q$ is empty, then certainly $Q$ is empty as well. Of course, the challenge is to construct a polytope $\mathcal{P}$ that is so close to $Q$ to increase the chances of successful detection of non-Hamiltonicity. In this paper, we construct such a set $\mathcal{P}$ -- determined by a polynomially-bounded number of linear constraints in continuous (relaxed) variables -- that appears to possess the preceding inclusion property for a vast majority of non-Hamiltonian cubic graphs.

The approaches given in this manuscript differ from the more traditional approach of identifying facet-inducing cuts in that we attempt to determine {\em non-Hamiltonicity} by forcing $\mathcal{P}$ to become empty, rather than by seeking to eliminate subtours in an iterative fashion. Although the final approach given in this manuscript is an iterative procedure, none of the iterates are based upon the result of a previous iteration. Rather, for a given graph, we will outline a series of tests that can be identified in advance, and the failure of any of those tests guarantees non-Hamiltonicity.

Cubic graphs represent a natural test laboratory for this methodology, not only because HCP is still NP-complete, but also because of the availability of reliable public-domain generators (e.g. see Meringer \cite{genreg}) that are capable of efficiently enumerating all nonequivalent connected cubic graphs of a given size. Our construction of the polyhedron $\mathcal{P}$ is achieved in two stages. First, we propose a \lq\lq base model", containing constraints which are generic for all graphs. Secondly, we add in additional constraints in an iterative fashion whenever certain structures are present in the graph. Importantly, however, these structures can be identified by preprocessing requiring only polynomial-time algorithms.

Specifically, the structures that we search for are those that identify \lq\lq brittle points" in the graph. Recently, in Baniasadi et al \cite{genetictheory,genetictheory2}, it was demonstrated that the set of all connected cubic graphs can be separated into two disjoint subsets, namely {\em genes} and {\em descendants}. The key distinction between these two subsets is the presence (or absence) of special edge cut sets known as {\em cubic crackers}. The following two definitions are paraphrased from Baniasadi et al. \cite{genetictheory2}.

\begin{definition}In a cubic graph, a {\em $k$-cracker} $c_k$ is an edge cut set of cardinality $k$ containing no adjacent edges, such that no proper subset of $c_k$ is also an edge cut set. A {\em cubic cracker} is a cracker with cardinality no greater than 3.\label{def-crackers}\end{definition}

\begin{definition}A {\em gene} is a connected cubic graph that contains no cubic crackers, and a descendant is a connected cubic graph that contains at least one cubic cracker.\label{def-gene}\end{definition}

In Baniasadi et al. \cite{genetictheory,genetictheory2}, crackers were used to develop a decomposition theory for cubic graphs that exploits genes and crackers. That theory does not include any algorithmic results for identifying non-Hamiltonicity. However, the results of this manuscript provide strong empirical evidence that cubic crackers contain much information about the Hamiltonicity of a graph. Specifically, consideration
of the cubic crackers can very often be used to detect {\em non-Hamiltonicity}, if it is present. The method presented in this manuscript
either finds that a cubic graph is definitely non-Hamiltonian, or returns an inconclusive result. Since the majority of connected cubic graphs are Hamiltonian \cite{robinson}, the latter outcome is obviously the most common, and by itself does not provide certainty that the graph is Hamiltonian. However, we will demonstrate that the number of false positives (that is, the number of non-Hamiltonian graphs returning an inconclusive result) is extremely low whenever our iterative procedure is used. Since determining whether a polyhedron, defined by polynomially-many linear constraints in polynomially-many continuous variables is empty, can be done by linear programming, this is a problem of polynomial complexity. Hence, the results presented in this manuscript serve to suggest that polynomial algorithms may be able to correctly identify non-Hamiltonicity of cubic graphs in the vast majority of cases.

Since non-Hamiltonian genes, by definition, do not contain any cubic crackers, the constraints based on the latter cannot be expected to be successful for these graphs. Indeed, they return an inconclusive result in all cases tested to date. Such non-Hamiltonian genes are called {\em mutants}. It was conjectured in \cite{genetictheory2} that mutants are extremely rare - indeed, only three such graphs containing 18 or fewer vertices exist. Despite their rarity, their identification may still be considered important; one justification could be that the set of mutants is a superset to the more famous set of (nontrivial) \emph{Snarks} \cite{gardner}. For this reason an additional heuristic is suggested at the end of this manuscript that succeeds in correctly identifying two of the three aforementioned graphs as non-Hamiltonian.

A famous related approach is the, now classical, Dantzig-Fulkerson-Johnson relaxation of the TSP with subtour constraints (e.g., see Dantzig et al. \cite{dantzig} and Cook et al \cite{cook}), which we will henceforth refer to as the DFJ relaxation. The constraints of that model, defined in terms of a given graph, also induce a polyhedron $\mathcal{P}_c \supset \mathcal{Q}$. Then, if the polyhedron $\mathcal{P}_c$ is empty, the graph is definitely non-Hamiltonian. Although there are exponentially many subtour constraints, it is possible to determine whether the polyhedron is empty in polynomial time using cutting plane techniques. We will compare the performance of the DFJ relaxation with the method proposed in this manuscript, in terms of the proportion of non-Hamiltonian graphs identified.

Three previous attempts to construct a desirable polyhedron $\mathcal{P} \supset Q$ were included in the recent PhD theses of Haythorpe \cite{mikethesis} and Eshragh \cite{alithesis}, and in Avrachenkov et al \cite{avrachenkov}. In Haythorpe \cite{mikethesis}, two of these polyhedra $\mathcal{P}$ in variables corresponding to arcs were constructed, it was conjectured (based on empirical evidence) that both polyhedra are empty for any cubic bridge graph. Bridge graphs are always non-Hamiltonian, and it was conjectured in \cite{conjecturepaper} that, asymptotically, almost all non-Hamiltonian graphs are bridge graphs.
However, no other non-Hamiltonian graphs were detected using either of the polyhedra in Haythorpe \cite{mikethesis}.

In Eshragh \cite{alithesis}, and later in Avrachenkov et al \cite{avrachenkov}, a related but somewhat different polyhedron was constructed. New variables corresponding not only to arcs in a graph, but also to positional information, were introduced. These variables have a probabilistic interpretation - a variable $x^k_{r,ia}$ can be thought of as the probability that arc $(i,a)$ is selected at the $r$-th step in a Hamiltonian cycle beginning at vertex $k$. By convention, if vertex $a$ is the first vertex visited in the Hamiltonian cycle after vertex $k$, then arc $(k,a)$ is said to be the $0$-th step of that Hamiltonian cycle. Clearly, in a solution corresponding to a Hamiltonian cycle, the variables must either take values 0 or 1. However, to retain linearity, this binary requirement was relaxed to allow $x^k_{r,ia}$ to simply take continuous values in $[0,1]$. The polyhedron given in these variables was again demonstrated empirically to be empty for any cubic bridge graph. In addition, a small number of non-bridge non-Hamiltonian graphs also generated empty polyhedra. These results motivated the further analysis and development of this approach, detailed in this manuscript.

This paper is structured as follows. In Section \ref{sec-bm} we prove that our constructed polyhedron is empty whenever induced by a cubic bridge graph. We also demonstrate empirically that this polyhedron constitutes an LP relaxation that is at least as strong as the DFJ relaxation when induced by cubic graphs. In Section \ref{sec-cc} we introduce additional constraints based on cubic crackers that permit us to correctly identify the Hamiltonicity of nearly all tested cubic graphs (over 40,000). By contrast, when an equivalent approach based on the DFJ relaxation was applied, very little improvement over their standard model was detected. It thus appears that the variables $x^k_{r,ia}$ can be used to capture valuable positional information that is not so easily captured by the variables used in the classical models. Finally, in the Appendix, we supply a list of further constraints which did not prove useful in our experimentation but may, nonetheless, be useful in further reducing the constructed polyhedron.

\section{Base HCP Feasibility Problem}\label{sec-bm}

Consider a connected undirected cubic graph $\Gamma$ containing $N$ vertices. For ease of notation, we denote the set of vertices that can be reached from a given vertex $i$ in a single step by $\mathcal{A}(i)$. So if arc $(i,a)$ exists in $\Gamma$, we say that $a \in \mathcal{A}(i)$. Similarly to the model introduced in Eshragh \cite{alithesis} and Avrachenkov et al \cite{avrachenkov}, we shall work in the space of vectors ${\bf x} = \{x^k_{r,ia} \quad | \quad k = 1, \hdots, N;\;\;r = 0, \hdots, N-1;\;\;i = 1, \hdots, N;\;\;a \in \mathcal{A}(i)\}.$ If $h$ is a given Hamiltonian cycle in $\Gamma$, it uniquely identifies a vector ${\bf x}_h$ whose entries $x^k_{r,ia}$ are equal to 1 whenever $(i,a)$ is the $r$-th arc on $h$, assuming $k$ to be the starting vertex, and all other entries $x^k_{r,ia}$ are equal to 0. The convex hull of such ${\bf x}_h$ (corresponding to all Hamiltonian cycles of $\Gamma$) forms the polytope $Q$ introduced in the previous section. Of course, we do not want to demand that entries of ${\bf x}$ are binary, and instead work with the relaxation that $0 \leq x^k_{r,ia} \leq 1$; $\forall k, r, i, a$. Hence, if we also demand that $\sum\limits_{i = 1}^N\sum\limits_{a \in \mathcal{A}(i)} x^k_{r,ia} = 1$, then the variable $x^k_{r,ia}$ can be interpreted as the probability that arc $(i,a)$ is selected as the $r$-th step in a Hamiltonian cycle beginning at vertex $k$. Although these variables allow for 0-1 probability corresponding to $N$ different Hamiltonian cycles (one starting from each vertex), the intention is for the probabilities to correspond to a single Hamiltonian cycle $N$ times, with the different starting vertex being the only distinction.

We now introduce the base polyhedron $\mathcal{P}_b \supset Q$, defined as the set of feasible solutions to the following family of linear constraints:

\begin{eqnarray}\sum_{a \in \mathcal{A}(i)} x^k_{r,ia} - \sum_{a \in \mathcal{A}(i)} x^k_{r-1,ai} & = & 0, \qquad \forall k = 1, \hdots, N;\quad r = 1, \hdots, N-1;\quad i = 1, \hdots, N,\nonumber\\
& & \label{c1}\\
\sum_{a \in \mathcal{A}(i)} x^k_{r,ia} - \sum_{a \in \mathcal{A}(k)} x^i_{N-r,k,a} & = & 0, \qquad \forall k = 1, \hdots, N;\quad r = 1, \hdots, N-1;\quad i = 1, \hdots, N,\nonumber\\
 & & \label{ex1}\\
\sum_{r = 0}^{N-1} x^k_{r,ia} - \sum_{r = 0}^{N-1} x^j_{r,ia} & = & 0, \qquad \forall j,k = 1, \hdots, N, j \neq k;\quad (i,a) \in \Gamma,\label{c2}\\
\sum_{k = 1}^N x^k_{r,ia} - \sum_{k = 1}^N x^k_{s,ia} & = & 0, \qquad \forall r,s = 0, \hdots, N-1, r \neq s;\quad (i,a) \in \Gamma,\label{ex2}\\
\sum_{r = 0}^{N-1} \sum_{a \in \mathcal{A}(i)} x^k_{r,ia} & = & 1, \qquad \forall k = 1, \hdots, N;\quad i = 1, \hdots, N,\label{c3}\\
\sum_{k = 1}^N \sum_{a \in \mathcal{A}(i)} x^k_{r,ia} & = & 1, \qquad \forall r = 0, \hdots, N-1;\quad i = 1, \hdots, N,\label{c4}\\
x^k_{0,ia} & = & 0, \qquad \forall k = 1, \hdots, N;\quad i \neq k;\quad (i,a) \in \Gamma,\label{c5}\\
x^k_{r,ia} & \geq  & 0, \qquad \forall k = 1, \hdots, N;\quad r = 0, \hdots, N-1;\quad (i,a) \in \Gamma.\label{c6}\end{eqnarray}

The next proposition shows that constraints (\ref{c1})--(\ref{c6}) are consistent whenever $\Gamma$ is Hamiltonian.

\begin{proposition}If $\Gamma$ is Hamiltonian, $\mathcal{P}_b \neq \emptyset$. Indeed, ${\bf x}_h \in \mathcal{P}_b$ whenever $h$ is a Hamiltonian cycle in $\Gamma$.\label{prop-ham_satis}\end{proposition}

\begin{proof}Consider a Hamiltonian cycle $h$ in the graph. Corresponding to this Hamiltonian cycle is a setting of the variables ${\bf x}_h$ containing entries $x^k_{r,ia}$ such that each variable is either 1 or 0. Then, we interpret $\sum\limits_{a \in \mathcal{A}(i)} x^k_{r,ia}$ as the probability that vertex $i$ is the $r$-th vertex in a Hamiltonian cycle starting at vertex $k$ (by convention, we say that vertex $k$ is the $0$-th vertex in the cycle). Clearly for the Hamiltonian cycle, this expression will also be either 1 or 0.

We consider each constraint separately:

\begin{itemize}\item[$\bullet$] For any given constraint in (\ref{c1}), the first term is either 1 or 0. If it is 1, it implies that vertex $i$ is the $r$-th step in the Hamiltonian cycle starting at vertex $k$. Then, it is clear that vertex $i$ must be entered in the previous step, and therefore the second term must also be 1. Using similar reasoning, if the first term is 0, the second term must also be 0. In either case, the constraint is satisfied.

\item[$\bullet$] For any given constraint in (\ref{ex1}), the first term is either 1 or 0. If it is 1, it implies that vertex $i$ is the $r$-th step in the Hamiltonian cycle starting at vertex $k$. Then, it is clear that in a further $N-r$ steps vertex $k$ would be departed from again. Therefore, starting from vertex $i$, the cycle departs from vertex $k$ after $N-r$ steps, and therefore the second term must also be 1. Using similar reasoning, if the first term is 0, the second term must also be 0. Either way, the constraint is satisfied.

\item[$\bullet$] For any given constraint in (\ref{c2}), the first term corresponds to the probability that an arc $(i,a)$ is chosen at some stage in a Hamiltonian cycle starting at vertex $k$. Clearly, this term is either 1 or 0. If it is 1, then the same Hamiltonian cycle starting at a different vertex must still travel through this arc, and therefore the second term must also be 1. Using similar reasoning, if the first term is 0, the second term must also be 0. In either case, the constraint is satisfied.

\item[$\bullet$] For any given constraint in (\ref{ex2}), the first term corresponds to the probability that there is some starting vertex such that arc $(i,a)$ is the $r$-th step in a Hamiltonian cycle starting at that vertex. Clearly, this term is either 1 or 0. If it is 1, then for any other step $s$ there must be some starting vertex such that arc $(i,a)$ is the $s$-th step, and therefore the second term must also be 1. Using similar reasoning, if the first term is 0, the second term must also be 0. Either way, the constraint is satisfied.

\item[$\bullet$] For any given constraint in (\ref{c3}), the left hand side corresponds to the sum of probabilities that a particular vertex is selected at each possible step on a Hamiltonian cycle starting at a particular vertex. Since a Hamiltonian cycle passes through every vertex exactly once, this sum of probabilities must be 1, and the constraint is satisfied.

\item[$\bullet$] For any given constraint in (\ref{c4}), the left hand side corresponds to the sum of probabilities that a particular vertex is selected at a particular step on a Hamiltonian cycle starting at each possible starting vertex. Since the $N$ Hamiltonian cycles are all equivalent, this can only occur from a single starting vertex, but must occur once. Therefore, the sum of probabilities is 1, and the constraint is satisfied.

\item[$\bullet$] For any given constraint in (\ref{c5}), the left hand side corresponds to the probability that a vertex is departed at the very start of a Hamiltonian cycle starting at a different vertex. Clearly, this cannot happen, and so the probability is 0 and the constraint is satisfied.

\item[$\bullet$] For any given constraint in (\ref{c6}), we know that all the $x^k_{r,ia}$ corresponding to a Hamiltonian cycle are either 1 or 0, so they trivially satisfy the nonnegativity constraint.
\end{itemize}

Therefore, the setting of 0-1 values for ${\bf x}_h$ corresponding to any Hamiltonian cycle $h$ satisfies all the constraints. Then, any graph containing a Hamiltonian cycle is guaranteed to have feasible solutions.\end{proof}

Constraints (\ref{c1})--(\ref{c6}) represent obvious requirements of Hamiltonian cycles, however it is possible that $\mathcal{P}_b \neq \emptyset$ even when $\Gamma$ does not possess any Hamiltonian cycles. Indeed, for the famous 10-vertex Petersen graph \cite{holton} this is the case. However, it is natural to ask whether $\mathcal{P}_b = \emptyset$ for some natural subclass of non-Hamiltonian graphs. It is well-known that any graph containing a bridge is non-Hamiltonian, and such graphs appears to be the most common cubic non-Hamiltonian graphs \cite{conjecturepaper}. The following theorem proves that any graph containing a bridge is unable to satisfy the constraints. Consider a cubic bridge graph $\Gamma$ with $N$ vertices, containing a bridge $(u,v)$ that separates $\Gamma$ into two components $U$ and $V$, as displayed in Figure \ref{fig-bridge}. Clearly, $|U| + |V| = N$.

\vspace*{1cm}\begin{figure}[h!]\begin{center}\includegraphics[scale=0.5]{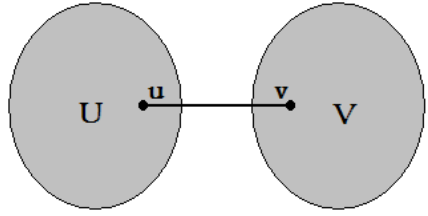}\caption{A bridge graph containing bridge $(u,v)$\label{fig-bridge}}\end{center}\end{figure}

\vspace*{0.5cm}\begin{theorem}Consider cubic graph $\Gamma$. Then, $\mathcal{P}_b = \emptyset$.\label{thm-bridge}\end{theorem}

\begin{proof}Let $(u,v)$ be the bridge in $\Gamma$ as defined above. Then, suppose there exists a feasible solution to (\ref{c1})--(\ref{c6}). From (\ref{c1}), selecting $k = u$, $r = 1$ and $i = v$,
\begin{eqnarray}\sum\limits_{a \in \mathcal{A}(v)} x^u_{1,va} = \sum\limits_{a \in \mathcal{A}(v)} x^u_{0,av}.\label{eq-1}\end{eqnarray}

From (\ref{c5}), selecting $k = u$, it is clear that
\begin{eqnarray}x^u_{0,av} & = & 0, \qquad \forall a \in \mathcal{A}(v), \mbox{when $a \neq u$.}\label{eq-2}\end{eqnarray}

Then, substituting (\ref{eq-2}) into (\ref{eq-1}), we obtain
\begin{eqnarray}\sum\limits_{a \in \mathcal{A}(v)} x^u_{1,va} & = & x^u_{0,uv}.\label{eq-3}\end{eqnarray}

Now, from (\ref{c4}), selecting $r = 0$ and $i = u$,
\begin{eqnarray}\sum\limits_{k=1}^N\sum\limits_{a \in \mathcal{A}(u)} x^k_{0,ua} & = & 1.\label{eq-4}\end{eqnarray}

From (\ref{c5}) it is clear that (\ref{eq-4}) simplifies to
\begin{eqnarray}\sum\limits_{a \in \mathcal{A}(u)} x^u_{0,ua} & = & 1.\label{eq-5}\end{eqnarray}

From (\ref{c3}), selecting $k = u$ and $i = u$,
\begin{eqnarray}\sum\limits_{r=0}^{N-1}\sum\limits_{a \in \mathcal{A}(u)} x^u_{r,ua} & = & 1,\label{eq-6}\end{eqnarray}

and from (\ref{c6}), selecting $k = u$ and $i = u$,
\begin{eqnarray}x^u_{r,ua} & \geq & 0, \qquad \forall r, \quad\forall a \in \mathcal{A}(u).\label{eq-7}\end{eqnarray}

Substituting (\ref{eq-7}) and (\ref{eq-5}) into (\ref{eq-6}), we conclude that
\begin{eqnarray}x^u_{r,ua} & = & 0, \qquad \forall r \geq 1, \quad \forall a \in \mathcal{A}(u).\label{eq-8}\end{eqnarray}

Using identical arguments to the above, but interchanging $u$ and $v$, we similarly conclude that
\begin{eqnarray}x^v_{r,va} & = & 0, \qquad \forall r \geq 1, \quad \forall a \in \mathcal{A}(v).\label{eq-9}\end{eqnarray}

Now, selecting $k = u$ and $i = u$ in (\ref{c1}) and substituting in (\ref{eq-8}), we obtain
\begin{eqnarray}\sum\limits_{a \in \mathcal{A}(u)}x^u_{r,au} & = & 0, \qquad \forall r \leq N-2, \quad\forall a \in \mathcal{A}(u),\label{eq-10}\end{eqnarray}

and by the nonnegativity of all variables we conclude that
\begin{eqnarray}x^u_{r,au} & = & 0, \qquad \forall r \leq N-2, \quad\forall a \in \mathcal{A}(u).\label{eq-11}\end{eqnarray}

Next, we consider the restricted set of vertices $c \in V$. From (\ref{c5}) we know that, since $u \in U$, $x^u_{0,ca} = 0$ for all $c \in V$.

Then, from (\ref{c1}) and (\ref{c5}), selecting $k = u$, $r = 1$, and $i = c$ we obtain

\begin{eqnarray}\sum_{a \in \mathcal{A}(c)} x^u_{1,ca} & = & \sum_{a \in \mathcal{A}(c)} x^u_{0,ac} = x^u_{0,uc}, \quad \forall c \in V.\label{eq-11.5}\end{eqnarray}

Since the only vertex in $V$ that $u$ connects to is $v$, we conclude from (\ref{eq-11.5}) and the nonnegativity of variables that $x^u_{1,ca} = 0$ for all $c \in V \setminus v$. Next, we consider $r \geq 2$. Summing (\ref{c1}) over all $c \in V$, and selecting $k = u$, we obtain
\begin{eqnarray}\sum\limits_{c \in V}\sum_{a \in \mathcal{A}(c)} x^u_{r,ca} & = & \sum\limits_{c \in V}\sum\limits_{a \in \mathcal{A}(c)} x^u_{r-1,ac}, \qquad \forall r \geq 2.\label{eq-12}\end{eqnarray}

The right hand side of (\ref{eq-12}) considers the sum of probabilities over all arcs that go to vertices in $V$ at a given starting vertex $u$ and stage $r-1$. However, with the sole exception of the bridge arc $(u,v)$, all arcs going to a vertex in $V$ originate in $V$. Similarly, with the sole exception of the bridge arc $(v,u)$, all arcs originating in $V$ go to a vertex in $V$. Therefore, the right hand side of (\ref{eq-12}) can alternatively be expressed as the sum of all arcs originating at (rather than going to) vertices in $V$ for the same fixed $k$ and $r$, with the only discrepency occurring at the bridge. Therefore, we can rewrite (\ref{eq-12}) as
\begin{eqnarray}\sum\limits_{c \in V}\sum_{a \in \mathcal{A}(c)} x^u_{r,ca} & = & \left[\sum\limits_{c \in V}\sum\limits_{a \in \mathcal{A}(c)} x^u_{r-1,ca}\right] + x^u_{r-1,uv} - x^u_{r-1,vu}, \qquad \forall r \geq 2.\label{eq-13}\end{eqnarray}

However, from (\ref{eq-8}) and (\ref{eq-11}), we know that both $x^u_{r-1,uv} = x^u_{r-1,vu} = 0$ for $r \geq 2$. Therefore, (\ref{eq-13}) simplifies to
\begin{eqnarray}\sum\limits_{c \in V}\sum\limits_{a \in \mathcal{A}(c)} x^u_{r,ca} & = & \sum\limits_{c \in V}\sum\limits_{a \in \mathcal{A}(c)} x^u_{r-1,ca}, \qquad \forall r \geq 2.\label{eq-14}\end{eqnarray}

Recursively substituting values of $r$ into (\ref{eq-14}) we obtain

\begin{eqnarray}\sum\limits_{c \in V}\sum\limits_{a \in \mathcal{A}(c)} x^u_{r,ca} & = & \sum\limits_{c \in V}\sum\limits_{a \in \mathcal{A}(c)} x^u_{1,ca}, \qquad \forall r \geq 2.\label{eq-14.5}\end{eqnarray}

Now, substituting (\ref{eq-3}) into the right hand side of (\ref{eq-14.5}) and recalling that $x^u_{1,ca} = 0$ for all $c \in V \setminus v$, we obtain
\begin{eqnarray}\sum\limits_{c \in V}\sum\limits_{a \in \mathcal{A}(c)} x^u_{1,ca} & = & x^u_{0,uv},\label{eq-15}\end{eqnarray}

and therefore, substituting (\ref{eq-15}) into (\ref{eq-14.5}), gives
\begin{eqnarray}\sum\limits_{c \in V}\sum\limits_{a \in \mathcal{A}(c)} x^u_{r,ca} & = & x^u_{0,uv}, \qquad \forall r \geq 1.\label{eq-15.5}\end{eqnarray}

Recalling that $x^u_{0,ca} = 0$ for all $c \in V$, we can sum over all values of $r$ in (\ref{eq-15.5}) to obtain
\begin{eqnarray}\sum\limits_{r = 0}^{N-1}\sum\limits_{c \in V}\sum\limits_{a \in \mathcal{A}(c)} x^u_{r,ca} & = & \sum\limits_{c \in V}\sum\limits_{a \in \mathcal{A}(c)} x^u_{0,ca} + \sum_{r = 1}^{N-1}\sum\limits_{c \in V}\sum\limits_{a \in \mathcal{A}(c)} x^u_{r,ca} = (N-1)x^u_{0,uv}.\label{eq-16}\end{eqnarray}

However, rearranging the above equation, we observe from (\ref{c3}) that
\begin{eqnarray}\sum\limits_{c \in V}\sum\limits_{r = 0}^{N-1}\sum\limits_{a \in \mathcal{A}(c)} x^u_{r,ca} & = & |V|.\label{eq-17}\end{eqnarray}

Equating (\ref{eq-16}) and (\ref{eq-17}) we see that
\begin{eqnarray}x^u_{0,uv} = \displaystyle\frac{|V|}{N-1}.\label{eq-18}\end{eqnarray}

Using identical arguments to the above, but interchanging vertices $c \in V$ with vertices $b \in U$, we similarly obtain
\begin{eqnarray}x^v_{0,vu} & = & \displaystyle\frac{|U|}{N-1}.\label{eq-19}\end{eqnarray}

Now, from (\ref{c2}), selecting $k = u$, $j = v$, $i = v$ and $a = u$,
\begin{eqnarray}\sum\limits_{r = 0}^{N-1} x^u_{r,vu} & = & \sum\limits_{r = 0}^{N-1} x^v_{r,vu}.\label{eq-20}\end{eqnarray}

Substituting (\ref{eq-9}) and (\ref{eq-19}) into (\ref{eq-20}), we obtain
\begin{eqnarray}\sum\limits_{r = 0}^{N-1} x^u_{r,vu} & = & x^v_{0,vu} = \displaystyle\frac{|U|}{N-1}.\label{eq-21}\end{eqnarray}

Substituting (\ref{eq-11}) into (\ref{eq-21}), the latter simplifies to
\begin{eqnarray}x^u_{N-1,vu} & = & \displaystyle\frac{|U|}{N-1}.\label{eq-22}\end{eqnarray}

Then, from (\ref{c3}), selecting $k = u$ and $i = v$,
\begin{eqnarray}\sum\limits_{r = 0}^{N-1}\sum\limits_{a \in \mathcal{A}(v)} x^u_{r,va} & = & 1.\label{eq-23}\end{eqnarray}

However, from (\ref{eq-3}), (\ref{eq-18}) and (\ref{eq-22}), we have
\begin{eqnarray}\sum\limits_{r=0}^{N-1}\sum\limits_{a \in \mathcal{A}(v)} x^u_{r,va} \geq \left[\sum\limits_{a \in \mathcal{A}(v)} x^u_{1,va}\right] + x^u_{N-1,vu} & = & \displaystyle\frac{|V|}{N-1} + \displaystyle\frac{|U|}{N-1} = \displaystyle\frac{N}{N-1}.\label{eq-24}\end{eqnarray}

Finally, substituting (\ref{eq-23}) into (\ref{eq-24}) we get $1 \geq \displaystyle\frac{N}{N-1}$. Since $N \geq 2$ this is a contradiction, and therefore no feasible solution to constraints (\ref{c1})--(\ref{c6}) exists for any cubic bridge graph $\Gamma$.\end{proof}

From Proposition \ref{prop-ham_satis} and Theorem \ref{thm-bridge} we know for certain that $\mathcal{P}_b \neq \emptyset$ for cubic Hamiltonian graphs, and $\mathcal{P}_b = \emptyset$ for cubic bridge graphs. However, for other non-Hamiltonian graphs, we have no such definite result. Constraints (\ref{c1})--(\ref{c6}) were implemented in CPLEX 12.3 and were tested on all non-Hamiltonian cubic graphs containing up to 18 vertices, to see if their non-Hamiltonicity could be identified. The results are given in Table \ref{tab-1}. Note that, although we only required constraints (\ref{c1}), (\ref{c2}), and (\ref{c3})--(\ref{c6}) to prove Theorem \ref{thm-bridge}, we still included constraints (\ref{ex1}) and (\ref{ex2}) in the model when computing the following results.

\begin{table}[h!]\scriptsize\begin{center}\begin{tabular}{|c||c|c||c|c|}\hline
$N$ & \# bridge graphs & \# solved & \# non-bridge non-Hamiltonian graphs & \# solved\\
\hline
10 & 1 & 1 & 1 & 0\\
\hline
12 & 4 & 4 & 1 & 0\\
\hline
14 & 29 & 29 & 6 & 1\\
\hline
16 & 186 & 186 & 33 & 6\\
\hline
18 & 1435 & 1435 & 231 & 42\\
\hline \end{tabular}\caption{Numbers of non-Hamiltonian graphs identified by the condition $\mathcal{P}_b = \emptyset$.\label{tab-1}}\end{center}\end{table}

Unsurprisingly, the results in Table \ref{tab-1} correctly identify precisely the same graphs as those reported in Avrachenkov et al \cite{avrachenkov}. As can be seen in the table, some non-bridge non-Hamiltonian graphs were identified by the condition $\mathcal{P}_b = \emptyset$, however the large majority of such graphs remain unidentified, suggesting that $\mathcal{P}_b \setminus Q$ is too large in general. In the following section, we propose a heuristic based on iteratively constructing new polyhedra by augmenting constraints (\ref{c1})--(\ref{c6}) with additional information in order to obtain vastly improved results.

However, first we compare the above results with the famous Dantzig-Fulkerson-Johnson relaxation of the TSP with subtour constraints, applied to HCP. The following model is paraphrased from Cook et al \cite{cook}, where variable $x_{ia}$ corresponds to the flow on the edge $(i,a)$ in a tour.

\begin{eqnarray}\sum\limits_{a \in \mathcal{A}(i)} x_{ia} & = & 2, \qquad \forall i = 1, \hdots, N,\label{eq-tsp1}\\
\sum\limits_{i \in S} \sum\limits_{a \in \mathcal{A}(i), a \not\in S} x_{ia} & \geq & 2, \qquad \forall 1 \leq |S| \leq N-1,\label{eq-tsp2}\\
x_{ia} & = & x_{ai}, \qquad \forall i = 1, \hdots, N;\quad a \in \mathcal{A}(i),\label{eq-tsp3}\\
x_{ia} & \leq & 1, \qquad \forall i = 1, \hdots, N;\quad a \in \mathcal{A}(i),\label{eq-tsp4}\\
x_{ia} & \geq & 0, \qquad \forall i = 1, \hdots, N;\quad a \in \mathcal{A}(i),\label{eq-tsp5}
\end{eqnarray}

where $S$ is a proper subset of the vertices of the graph. Let $\mathcal{P}_c$ denote the feasible region defined by (\ref{eq-tsp1})--(\ref{eq-tsp5}). Note that there are $2^N - 2$ constraints in (\ref{eq-tsp2}). Although cutting plane approaches have been successfully utilised in the past to solve (\ref{eq-tsp1})--(\ref{eq-tsp5}) in polynomial time, for the purposes of the following comparison, we included all of the constraints.

Much like for our base polyhedron $\mathcal{P}_b$, if $\mathcal{P}_c$ is empty, we can conclude that the graph is non-Hamiltonian. Constraints (\ref{eq-tsp1})--(\ref{eq-tsp5}) were implemented in CPLEX 12.3 and tested on the same set of non-Hamiltonian cubic graphs containing up to 18 vertices that were used to produce Table \ref{tab-1}. The results were quite surprising. Exactly the same set of graphs which were identified as being non-Hamiltonian in Table \ref{tab-1} were also identified as being non-Hamiltonian via the DFJ relaxation, and no additional graphs were able to be identified as non-Hamiltonian. Although these results were only obtained over relatively small cubic graphs, they are still striking enough to motivate the following conjecture.

\begin{conjecture}For any cubic graph $\Gamma$, (\ref{c1})--(\ref{c6}) are a polynomial-size set of constraints which are equivalent to the exponential-size set of constraints (\ref{eq-tsp1})--(\ref{eq-tsp5}). That is, $\mathcal{P}_b = \mathcal{P}_c.$\label{conj-basetsp}\end{conjecture}

Note that in the above conjecture we assume cubicity, since all of our test graphs were cubic. However, there are no constraints in either model that require or take advantage of cubicity, so it is reasonable to think that the conjecture, if true for cubic graphs, might also be true for a larger set of graphs.

\section{Cubic cracker HCP feasibility heuristic}\label{sec-cc}

The results in the previous section came from a generic model that is defined for all graphs. However, it seems sensible to take advantage of the specific structure of a given graph whenever possible. For this reason, we now investigate the use of cubic crackers to augment the set of constraints that define $\mathcal{P}_b$.

Recall from Definition \ref{def-crackers} that a cubic cracker is an edge cut set of cardinality no greater than three, containing no adjacent edges, such that no proper subset of the edges is an edge cut set. Identifying all cubic crackers in a graph is a simple process of considering all sets of up to three non-adjacent edges, and checking if their removal disconnects the graph. For a cubic graph there are $\displaystyle\frac{3N}{2}$ edges, so the number of such sets of three non-adjacent edges is bounded above by order $O(N^3)$. Confirming the connectedness of a graph is also a polynomial-time process (e.g. see Siek et al \cite{bgl}). Therefore, the process of identifying all cubic crackers is a polynomial-time process.

If no cubic crackers are identified in a cubic graph (that is, the graph is a gene) then no additional constraints are augmented by the approach outlined in this section. If any 1-crackers are identified in a graph, then by definition that graph is a bridge graph, and therefore non-Hamiltonian. Of course, from Theorem \ref{thm-bridge} we know that such graphs will be found to be infeasible by the condition $\mathcal{P}_b = \emptyset$, hence there is no need to solve a further feasibility problem for them.

If a cubic graph contains no 1-crackers, but contains some 2-crackers or 3-crackers (e.g. see Figure \ref{fig-crackers}), then we can augment constraints (\ref{c1})--(\ref{c6}) with new constraints. Since a cubic cracker contains no more than 3 edges, and the removal of these edges disconnects the graph, any Hamiltonian cycle must pass through exactly two of these edges. Starting on one side of the cubic cracker, the other side can only be accessed via an edge in the cubic cracker. Eventually, a second edge of the cubic cracker must be traversed in order to leave this second side and return to the starting vertex. However, once the second side has been exited, it is never returned to again (as there are not enough edges remaining to revisit the second side and still return to the starting vertex). Therefore the entirety of this second side must be traversed before a second edge in the cubic cracker is traversed. This knowledge allows us to impose certain structural demands on any Hamiltonian cycle in a cubic graph containing such cubic crackers (that is, a descendant graph). Call the two components either side of the cracker $U$ and $V$. Then clearly each edge of the cubic cracker joins a vertex in $U$ to a vertex in $V$.

\begin{figure}[h!]\begin{center}\includegraphics[scale=0.35]{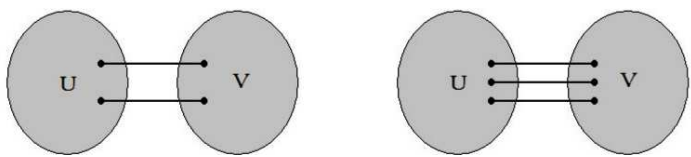} \caption{Graphs containing a 2-cracker and 3-cracker respectively\label{fig-crackers}}\end{center}\end{figure}

Each 2-cracker and 3-cracker in the graph is considered separately to see if their individual structural requirements induce infeasibility. For any given 2-cracker or 3-cracker, all possible selections of 2 edges within the cracker are considered. For 2-crackers there is only one choice, whereas for 3-crackers there are three different ways to select two edges. Then, for each selection of 2 edges, we fix them and demand that a Hamiltonian cycle be found that uses these fixed edges. Of course, the constraints in this manuscript are designed to identify {\em directed} Hamiltonian cycles, so we must select an orientation along the two edges - one edge will be traversed emanating from $U$, and the second edge back towards $U$. However, since all graphs considered in this manuscript are undirected, a directed Hamiltonian cycle in one direction exists if and only if a second directed Hamiltonian cycle traversing the same edges in the opposite direction also exists, and so we can select either orientation over the two selected edges. Since we are selecting orientations of edges, when no confusion is possible, we will simply refer to a selection of two arcs.

The {\em cubic cracker constraints} that will be given shortly make use of the following expression,
\begin{eqnarray}\sum_{r=0}^{N-1} r \sum_{a \in \mathcal{A}(i)} x^k_{r,ia},\label{eq-position_expression}\end{eqnarray}
which has a natural interpretation when the $x$ variables correspond to a Hamiltonian cycle. We know that, for a Hamiltonian cycle starting at vertex $k$, any given vertex $i$ must be departed from at some stage, and then never departed from again. So, for a Hamiltonian cycle, precisely one of the $x^k_{r,ia}$ terms in (\ref{eq-position_expression}) is 1, and the rest are 0. Since the term corresponding to the $r$-th step is multiplied by $r$, we can think of (\ref{eq-position_expression}) as the {\em position of vertex $i$ on the Hamiltonian cycle originating at vertex $k$}.

Suppose now that we have selected arcs $(u_1,v_1)$ and $(v_2,u_2)$ of a cubic cracker to be fixed, where $u_1, u_2 \in U$ and $v_1, v_2 \in V$. Then any Hamiltonian cycle that traverses these two arcs must satisfy the following four conditions:

\begin{enumerate}
  \item Vertex $v_1$ must come directly after $u_1$, unless $v_1$ is the starting vertex;
  \item Vertex $u_2$ must come directly after vertex $v_2$, unless $u_2$ is the starting vertex;
  \item Vertex $v_2$ must come exactly $|V|-1$ positions after $v_1$, unless $v_2$ comes before $v_1$;
  \item Vertex $u_1$ must come exactly $|U|-1$ positions after $u_2$, unless $u_1$ comes before $u_2$.
\end{enumerate}

Of course, there could be many more requirements imposed by the selection of two arcs of a cubic cracker, but we demonstrate shortly that the above four conditions are very powerful even by themselves.

The expression in (\ref{eq-position_expression}) allows us to augment the constraints that define $\mathcal{P}_b$ in order to define a new polytope $\mathcal{P}_c \subset \mathcal{P}_b$ that captures the above four conditions. Note however that $\mathcal{P}_c$ depends on the particular two fixed arcs of the cubic cracker considered, and so for a given graph there may be many such polyhedra $\mathcal{P}_c$. The augmenting constraints are:

\begin{eqnarray}\sum_{r=0}^{N-1} r\left[\sum_{a \in \mathcal{A}(v_1)} x^k_{r,v_1a} - \sum_{a \in \mathcal{A}(u_1)} x^k_{r,u_1a}\right] & = & 1, \qquad \forall k \neq v_1\label{eqc-1}\\
\sum_{r=0}^{N-1} r\left[\sum_{a \in \mathcal{A}(u_2)} x^k_{r,u_2a} - \sum_{a \in \mathcal{A}(v_2)} x^k_{r,v_2a}\right] & = & 1, \qquad \forall k \neq u_2\label{eqc-2}\\
\sum_{r=0}^{N-1} r\left[\sum_{a \in \mathcal{A}(v_2)} x^k_{r,v_2a} - \sum_{a \in \mathcal{A}(v_1)} x^k_{r,v_1a}\right] & = & |V|-1, \qquad \forall k \in U\label{eqc-3}\\
\sum_{r=0}^{N-1} r\left[\sum_{a \in \mathcal{A}(u_1)} x^k_{r,u_1a} - \sum_{a \in \mathcal{A}(u_2)} x^k_{r,u_2a}\right] & = & |U|-1. \qquad \forall k \in V\label{eqc-4}\end{eqnarray}

Note that the conditions on $k$ for constraint (\ref{eqc-3}) ensure that $v_2$ comes after $v_1$, and likewise the conditions on $k$ for constraint (\ref{eqc-4}) ensures that $u_1$ comes after $u_2$.

If, for a particular selection of two arcs of a cubic cracker, $\mathcal{P}_c = \emptyset$, then it is impossible to construct a Hamiltonian cycle that traverses those two arcs of the cubic cracker. If $\mathcal{P}_c = \emptyset$ for all possible selections of two arcs of a cubic cracker, then no Hamiltonian cycle can traverse the cubic cracker at all, and the graph is definitely non-Hamiltonian. If, however, $\mathcal{P}_c \neq \emptyset$ for some selection of 2 arcs of the cubic cracker, the next cubic cracker is considered instead. Only once all cubic crackers present in a graph have been considered, and we have found $\mathcal{P}_c \neq \emptyset$ for some selection of 2 arcs within each of them, is a final result of feasibility returned. Of course, in such a case, we still cannot be certain that the graph is Hamiltonian; however, as will be shown shortly, the incidence of false positives is extremely low.

Empirically, we have found that not all constraints (\ref{c1})--(\ref{c6}) are required for this approach to satisfy the condition $\mathcal{P}_c = \emptyset$. In particular, we found constraints (\ref{c3})--(\ref{c4}) were not required, and their presence in the model did not alter the result once cubic cracker constraints were used in all tested cases. Then, for each selection of 2 arcs $(u_1,v_1)$ and $(v_2,u_2)$ of a cubic cracker, the new feasibility problem solved is:

\begin{eqnarray}\sum_{a \in \mathcal{A}(i)} x^k_{r,ia} - \sum_{a \in \mathcal{A}(i)} x^k_{r-1,ai} & = & 0, \qquad \forall k = 1, \hdots, N; r = 1, \hdots, N-1; i = 1, \hdots, N,\nonumber\\
& & \label{eq3-1}\\
\sum_{a \in \mathcal{A}(i)} x^k_{r,ia} - \sum_{a \in \mathcal{A}(k)} x^i_{N-r,k,a} & = & 0, \qquad \forall k = 1, \hdots, N; r = 1, \hdots, N-1; i = 1, \hdots, N,\nonumber\\
& & \label{eq3-2}\\
\sum_{r = 0}^{N-1} x^k_{r,ia} - \sum_{r = 0}^{N-1} x^j_{r,ia} & = & 0, \qquad \forall j,k = 1, \hdots, N, j \neq k; (i,a) \in \Gamma,\label{eq3-3}\\
\sum_{k = 1}^N x^k_{r,ia} - \sum_{k = 1}^N x^k_{s,ia} & = & 0, \qquad r,s = 0, \hdots, N-1, r \neq s; (i,a) \in \Gamma,\label{eq3-3.5}\\
x^k_{0,ia} & = & 0, \qquad \forall k = 1, \hdots, N; i \neq k; (i,a) \in \Gamma,\label{eq3-4}\end{eqnarray}
\begin{eqnarray}
\sum_{r=0}^{N-1} r\left[\sum_{a \in \mathcal{A}(v_1)} x^k_{r,v_1,a} - \sum_{a \in \mathcal{A}(u_1)} x^k_{r,u_1a}\right] & = & 1, \qquad \forall k \neq v_1\label{eq3-5}\\
\sum_{r=0}^{N-1} r\left[\sum_{a \in \mathcal{A}(u_2)} x^k_{r,u_2a} - \sum_{a \in \mathcal{A}(v_2)} x^k_{r,v_2a}\right] & = & 1, \qquad \forall k \neq u_2\label{eq3-6}\\
\sum_{r=0}^{N-1} r\left[\sum_{a \in \mathcal{A}(v_2)} x^k_{r,v_2a} - \sum_{a \in \mathcal{A}(v_1)} x^k_{r,v_1a}\right] & = & |V|-1, \qquad \forall k \in U\label{eq3-7}\\
\sum_{r=0}^{N-1} r\left[\sum_{a \in \mathcal{A}(u_1)} x^k_{r,u_1a} - \sum_{a \in \mathcal{A}(u_2)} x^k_{r,u_2a}\right] & = & |U|-1, \qquad \forall k \in V\label{eq3-8}\end{eqnarray}
\begin{eqnarray}
x^k_{r,ia} & \geq  & 0, \qquad \forall k = 1, \hdots, N; r = 0, \hdots, N-1; (i,a) \in \Gamma.\label{eq3-9}
\end{eqnarray}

We refer to this systematic approach to identifying non-Hamiltonicity by testing the condition that $\mathcal{P}_c = \emptyset$ over all selections of 2 arcs for each cubic cracker as the {\em cubic cracker HCP feasibility heuristic (CHFH)}. Although the approach of considering individual crackers in isolation is weaker than considering all cubic crackers in conjunction, it is taken in order to ensure the number of feasibility problems considered remains bounded polynomially. A cubic graph may contain many 3-crackers, and for each of these, three orientations need to be considered. For example, if a graph contained $m$ 3-crackers, considering each selection of 2 arcs of the 3-crackers in conjunction would require $3^m$ feasibility problems, whereas considering them separately requires only $3m$ feasibility problems.

CHFH was implemented in CPLEX 12.3 and all non-Hamiltonian cubic graphs containing up to 18 vertices were tested. The results are given in Table \ref{tab-3}, where the \lq\lq Improvement" column refers to the number of correctly identified non-bridge non-Hamiltonian graphs that were not so identified by the model in Section \ref{sec-bm}.

\begin{table}[h!]\scriptsize\begin{center}\begin{tabular}{|c||c|c||c|c|c|}\hline
$N$ & \# bridge graphs & \# solved & \# non-bridge non-Hamiltonian graphs & \# solved & Improvement\\
\hline
10 & 1 & 1 & 1 & 0 & 0\\
\hline
12 & 4 & 4 & 1 & 1 & 1\\
\hline
14 & 29 & 29 & 6 & 6 & 5\\
\hline
16 & 186 & 186 & 33 & 33 & 27\\
\hline
18 & 1435 & 1435 & 231 & 228 & 186\\
\hline \end{tabular}\caption{Numbers of non-Hamiltonian graphs identified by CHFH.\label{tab-3}}\end{center}\end{table}

As can be seen in Table \ref{tab-3}, almost all non-Hamiltonian graphs are identified by CHFH. In fact, only a single non-Hamiltonian descendant returns feasibility in all the graphs tested. The other non-Hamiltonian graphs that return feasibility are the three mutants (non-Hamiltonian genes) containing up to 18 vertices. This represents only four graphs from the total set of 1,927 non-Hamiltonian connected cubic graphs containing up to 18 vertices that this approach is unable to identify. Although we provide no formal results to prove that this high success rate should continue as larger graphs are considered, the empirical results to this stage are very compelling.

Once again, we chose the Dantzig-Fulkerson-Johnson relaxation of the TSP with subtour constraints to benchmark the above results. In a similar manner to CHFH, we constructed a cubic cracker-based heuristic for constraints (\ref{eq-tsp1})--(\ref{eq-tsp5}). The heuristic was similar to CHFH in that we first identified all cubic crackers in a graph, and then for each one found, we iteratively considered all possible choices of fixing two edges to be traversed. For 2-crackers, the subtour constraints (\ref{eq-tsp2}) ensure that this occurs anyway, but in the case of 3-crackers, this was achieved by simply removing the unused edge from the model in each iteration. This process is equivalent to enforcing constraints (\ref{eq3-5})--(\ref{eq3-6}). However, since there is no way to specify at which step vertices are visited in the DFJ relaxation, there are no equivalent constraints to (\ref{eq3-7})--(\ref{eq3-8}).

This heuristic was also implemented in CPLEX 12.3 and all non-Hamiltonian cubic graphs containing up to 18 vertices were tested. The results are given in Table \ref{tab-4}.

\begin{table}[h!]\scriptsize\begin{center}\begin{tabular}{|c||c|c||c|c|c|}\hline
$N$ & \# bridge graphs & \# solved & \# non-bridge non-Hamiltonian graphs & \# solved & Improvement\\
\hline
10 & 1 & 1 & 1 & 0 & 0\\
\hline
12 & 4 & 4 & 1 & 0 & 0\\
\hline
14 & 29 & 29 & 6 & 1 & 0\\
\hline
16 & 186 & 186 & 33 & 6 & 0\\
\hline
18 & 1435 & 1435 & 231 & 46 & 4\\
\hline \end{tabular}\caption{Numbers of non-Hamiltonian graphs identified by the cubic cracker-based heuristic built on the DFJ relaxation.\label{tab-4}}\end{center}\end{table}

As can be seen, the cubic cracker-based heuristic built on the DFJ relaxation does identify a few more non-Hamiltonian graphs than the standard DFJ relaxation (specifically, four additional 18-vertex graphs), but the majority of them remain unidentified. It seems that the additional flexibility offered by the variables used in CHFH permits a far greater number of tested non-Hamiltonian graphs to be identified. These results demonstrate that although cubic crackers contain very useful information, how best to take advantage of that information also requires careful consideration.

The one descendant graph that provided a false positive for CHFH can be seen in Figure \ref{fig-18_4652}. It contains 18 vertices. When generated by GENREG \cite{genreg} using the command {\texttt{genreg 18 3 3}} it has the ID \#4652. For this reason we refer to it as $\Gamma^{18}_{4652}$. It contains a single 2-cracker, and no other cubic crackers.

\vspace*{1cm}\begin{figure}[h!]\begin{center}\includegraphics[scale=0.35]{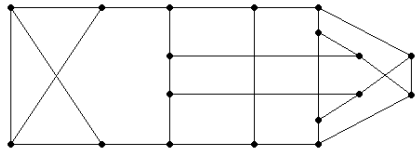}\caption{Non-Hamiltonian 18-vertrex graph $\Gamma^{18}_{4652}$.\label{fig-18_4652}}\end{center}\end{figure}

The other three non-Hamiltonian cubic graphs that return feasibility for CHFH are well-known in literature. They are the Petersen graph, and the two Blanu\u{s}a snarks, which are displayed in Figure \ref{fig-snarks}. CHFH has little opportunity to identify these graphs as non-Hamiltonian because they contain no cubic crackers. However, there is a logical approach that can be pursued to attempt to identify such graphs as non-Hamiltonian. For any cubic graph $\Gamma$, a second graph $\Gamma_2$ can be formed by replacing a vertex $v \in \Gamma$ by a triangle, such that for each edge in $\Gamma$ that is adjacent to $v$, there is an equivalent edge in $\Gamma_2$ that is adjacent to one of the vertices of the triangle\footnote{The operation of replacing a vertex by a triangle is known in engineering fields as a Y--$\triangle$ transform \cite{kennelly}.}. Then, $\Gamma_2$ is Hamiltonian if and only if $\Gamma$ is Hamiltonian. However, the introduction of a triangle creates a 3-cracker in $\Gamma_2$, and so CHFH can be used for $\Gamma_2$. If $\Gamma_2$ is found to be non-Hamiltonian by CHFH, then $\Gamma$ is certainly non-Hamiltonian as well. This approach was tested for the Petersen graph and the two Blanu\u{s}a snarks, with the triangle replacing the first vertex in each graph as produced by GENREG \cite{genreg}. In two of the three cases, specifically those of the Petersen graph and the first Blanu\u{s}a snark, CHFH was able to identify the descendants obtained as non-Hamiltonian. Only for the second Blanu\u{s}a snark did CHFH fail to identify the descendant obtained as non-Hamiltonian via this approach. Therefore, using this systematic approach, the Hamiltonicity of only two graphs (out of 45,982 connected cubic graphs tested) is not correctly identified.

\vspace*{1cm}\begin{figure}[h!]\begin{center}\includegraphics[scale=0.5]{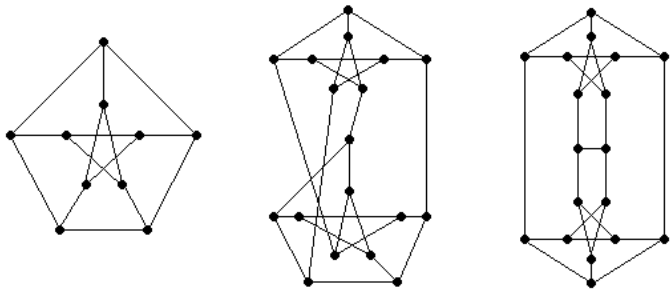}\caption{The three smallest snarks - the Petersen graph, and the two Blanu\u{s}a snarks.\label{fig-snarks}}\end{center}\end{figure}

Of course, it is entirely possible that there are additional constraints that could be used in CHFH that may result in the correct identification of these two exceptional cases. In the appendix we list all the additional constraints we have used to date, but note that these constraints did not result in the successful identification of $\Gamma^{18}_{4652}$ nor the second Blanu\u{s}a snark as non-Hamiltonian graphs by this approach.

\section*{Appendix: Additional constraints}

We separate the appendix into three smaller sections; extra constraints that could be applied in the base feasibility program, constraints in new variables that represent arcs in a reverse Hamiltonian cycle and arcs not selected in a Hamiltonian cycle, and additional cubic cracker constraints. It should be noted that in all tested graphs, augmenting these constraints to CHFH does not alter the output of CHFH, however it is possible that it would do so for some larger cubic graphs. This appendix is not intended as an exhaustive list of such constraints that could be used, but rather a guide to interested readers as to which constraints have already been tested. Note that some of these constraints may be redundant given other constraints. A short interpretation is given after each constraint.

Also note that in the following additional constraints we occasionally use $\mu^i_j$, defined to be the length of the shortest path from vertex $i$ to vertex $j$. These can be computed, in advance, in quadratic time using Dijkstra's algorithm.

{\underline{\textbf{Basic constraints:}}}

\begin{eqnarray}\sum_{i=1}^N\sum_{a \in \mathcal{A}(i)} x^k_{r,ia} & = & 1, \qquad \forall k = 1, \hdots, N; r = 0, \hdots, N-1.\end{eqnarray}
Interpretation: For any given starting vertex, at any given stage, exactly one arc must be selected.
\begin{eqnarray}x^k_{r,ka} & = & 0, \qquad \forall k = 1, \hdots, N; r = 1, \hdots, N-1; a \in \mathcal{A}(k).\end{eqnarray}
Interpretation: The starting vertex can not be departed from again after the initial step.
\begin{eqnarray}x^k_{r,bk} & = & 0, \qquad \forall k = 1, \hdots, N; r = 0, \hdots, N-2; b \in \mathcal{A}(k).\end{eqnarray}
Interpretation: The starting vertex can not be reentered until the final step in the Hamiltonian cycle.
\begin{eqnarray}x^k_{r,ia} & = & 0, \qquad \forall k = 1, \hdots, N; (i,a) \in \Gamma, r < \mu_i^k\mbox{ or } r > N - \mu_a^k - 1.\end{eqnarray}
Interpretation: An edge $(i,a)$ cannot be traversed in a certain step unless it is both possible to reach the vertex $i$ by that time step, and then possible to return to the starting vertex from vertex $a$ in the remaining time steps.
\begin{eqnarray}\sum_{r=0}^{N-1} r \left[\sum_{a \in \mathcal{A}(i)} x^k_{r,ia} + \sum_{a \in \mathcal{A}(j)} x^k_{r,ja}\right] & \geq & 2\min\left(\mu^k_i,\mu^k_j\right) + \mu^i_j,\nonumber\\
 & & \;\forall k = 1, \hdots, N; i,j = 1, \hdots, N \setminus k.\end{eqnarray}
Interpretation: Let vertices $i$ and $j$ be on a Hamiltonian cycle starting at vertex $k$. One must be reached before the other. Then, once the closer vertex is reached, the further vertex cannot be reached in fewer steps than the shortest path between vertices $i$ and $j$.

{\underline{\textbf{New variables:}}}

In this category we introduce two sets of new variables denoted by $y^k_{r,ia}$ and $s^k_{ia}$. The former, $y^k_{r,ia}$ are equivalent to the variables in ${\bf x}$ used so far, except that it corresponds to the {\em reverse Hamiltonian cycle}. That is, $y^k_{r,ia} = 1$ if arc $(i,a)$ is the $r$-th step in the reverse Hamiltonian cycle originating at vertex $k$, and $y^k_{r,ia} = 0$ otherwise. The intention is that the Hamiltonian cycles traced out by the variables in ${\bf x}$ and ${\bf y}$ are identical except for their direction. Obviously then, $x^k_{r,ia} = y^k_{N-r-1,ai}$. Every constraint given so far for the variables in ${\bf x}$ should also be applied to the variables in ${\bf y}$. In addition, we supply sets of coupling constraints below.

Since we consider cubic graphs in this manuscript, from any given vertex there are three arcs emanating. For a Hamiltonian cycle, one will be selected for the variables in ${\bf x}$ at some stage $r$, a second will be selected for the variables in ${\bf y}$ at stage $N-r-1$, and a third is not selected at any stage. The new variable $s^k_{ia}$ corresponds to the arc never selected in the following manner: $s^k_{ia} = 0.5$ if arc $(i,a)$ is not selected in the Hamiltonian cycle or the reverse Hamiltonian cycle and either $k = i$ or $k = a$, and $s^k_{ia} = 0$ otherwise. Note that for the $s^k_{ia}$ variables, no subscript corresponding to a stage is required, however a superscript corresponding to the starting vertex is used to enable the simple construction of coupling constraints between the variables in ${\bf x}$, ${\bf y}$ and ${\bf s}$. Note that for a given $k$, exactly two $s^k_{ia}$ are non-zero.

The following constraints involve the new (and sometimes old) variables.

\begin{eqnarray}\sum_{a \in \mathcal{A}(i)} x^k_{r,ia} - \sum_{a \in \mathcal{A}(i)} y^k_{N-r,ia} & = & 0, \qquad \forall k = 1, \hdots, N; r = 1, \hdots, N-1; i = 1, \hdots, N.\nonumber\\
& & \end{eqnarray}
Interpretation: If a vertex is departed from at stage $r$ in a Hamiltonian cycle, then it must be entered at stage $N-r-1$ in the reverse Hamiltonian cycle, and therefore departed from at stage $N-r$ in the reverse Hamiltonian cycle.
\begin{eqnarray}\sum_{r=0}^{N-1} r \left[\sum_{a \in \mathcal{A}(i)} x^k_{r,ia} + \sum_{a \in \mathcal{A}(i)} y^k_{r,ia}\right] & = & N, \qquad \forall k = 1, \hdots, N; i = 1, \hdots, N; i \neq k.\nonumber\\
& & \end{eqnarray}
Interpretation: If the position of any vertex (other than the starting vertex) on a Hamiltonian cycle is $\ell$, then the position of the same vertex on the reverse Hamiltonian cycle must be $N-\ell$, and so their sum is $N$.
\begin{eqnarray}\sum_{k=1}^N x^k_{r,ia} + \sum_{k=1}^N y^k_{r,ia} & \leq & 1, \qquad \forall r = 0, \hdots, N-1; (i,a) \in \Gamma.\label{eq-newvar}\end{eqnarray}
Interpretation: If an arc $(i,a)$ is used at stage $r$ in a Hamiltonian cycle or reverse Hamiltonian cycle, it can only be so in one of them and for only one starting vertex. Of course, arc $(i,a)$ may be used by neither, in which case the left hand side of (\ref{eq-newvar}) is 0.
\begin{eqnarray}\sum_{a \in \mathcal{A}(i)} s^k_{ia} - \sum_{a \in \mathcal{A}(i)} s^k_{ai} & = & 0, \qquad \forall k = 1, \hdots, N; i = 1, \hdots, N,\\
\sum_{k=1}^N\sum_{a \in \mathcal{A}(i)} s^k_{ia} & = & 1, \qquad \forall i = 1, \hdots, N,\\
\sum_{i=1}^N\sum_{a \in \mathcal{A}(i)} s^k_{ia} & = & 1, \qquad \forall k = 1, \hdots, N.\end{eqnarray}
Interpretation: For any given starting vertex, there is a single incoming and a single outgoing arc not used in either Hamiltonian cycle.
\begin{eqnarray}s^k_{ia} & = & 0, \qquad \forall k = 1, \hdots, N; (i,a) \in \Gamma; \mu^k_i > 1.\end{eqnarray}
Interpretation: Ensures that we only consider legal arcs in the $s$ variables.
\begin{eqnarray}\sum_{a \in \mathcal{A}(k)} s^k_{ka} & = & 0.5, \qquad \forall k = 1, \hdots, N,\\
\sum_{a \in \mathcal{A}(i)} s^k_{ia} & \leq & 0.5, \qquad \forall k = 1, \hdots, N; i = 1, \hdots, N; i \neq k.\end{eqnarray}
Interpretation: For a given vertex, there is a single outgoing arc not used in either the forward or the reverse Hamiltonian cycle.
\begin{eqnarray}x^j_{r,ia} + y^k_{r,ia} & \leq & 1, \qquad \forall k = 1, \hdots, N; \forall j = 1, \hdots, N; r = 0, \hdots, N-1; (i,a) \in \Gamma.\nonumber\\
& & \end{eqnarray}
Interpretation: A given arc cannot appear in both the Hamiltonian cycle and the reverse Hamiltonian cycle.
\begin{eqnarray}\sum_{r=0}^{N-1} x^k_{r,ia} - \sum_{r=0}^{N-1} y^k_{r,ai} & = & 0, \qquad \forall k = 1, \hdots, N; (i,a) \in \Gamma.\end{eqnarray}
Interpretation: If an arc is selected in a Hamiltonian cycle at some stage, its reverse must be selected in the reverse Hamiltonian cycle at some stage.
\begin{eqnarray}\sum_{k=1}^N x^k_{r,ia} + \sum_{k=1}^N y^k_{r,ia} + \sum_{k=1}^N s^k_{ia} & = & 1, r = 0, \hdots, N-1; \qquad \forall (i,a) \in \Gamma.\end{eqnarray}
Interpretation: Any arc in the graph is either used in a Hamiltonian cycle, a reverse Hamiltonian cycle, or not at all.

{\underline{\textbf{Cracker constraints:}}}

Consider a cubic cracker that divides a graph into two components $U$ and $V$, containing $|U|$ and $|V|$ vertices respectively. Assume, without loss of generality, that $|U| \leq |V|$. Suppose two arcs $(u_1,v_1)$ and $(v_2,u_2)$ from the cubic cracker are selected to be traversed by a Hamiltonian cycle. The following additional cracker constraints can be applied for this selection of two arcs. In the following, $\delta^i_j$ is the Kronecker delta that is equal to 1 if $i = j$, and 0 otherwise.

\begin{eqnarray}\sum_{r=0}^{N-1} r \sum_{a \in \mathcal{A}(u_1)} x^{v_1}_{r,u_1a} & = & N-1,\\
\sum_{r=0}^{N-1} r \sum_{a \in \mathcal{A}(w_2)} x^{u_2}_{r,v_2a} & = & N-1.\end{eqnarray}
Interpretation: A Hamiltonian cycle starting at the head vertex of a cracker arc must visit the tail vertex of the cracker after $N-1$ steps.
\begin{eqnarray}x^k_{r,u_1a} & = & 0, \qquad \forall k = 1, \hdots, N; r = 0, \hdots, N-1; a \neq v_1,\\
x^k_{r,v_2a} & = & 0, \qquad \forall k = 1, \hdots, N; r = 0, \hdots, N-1; a \neq u_2.\end{eqnarray}
Interpretation: The tail vertex of a cracker arc must be followed by the head vertex of the cracker arc.
\begin{eqnarray}x^k_{r,v_1u_1} & = & 0, \qquad \forall k = 1, \hdots, N; r = 0, \hdots, N-1,\\
x^k_{r,u_2v_2} & = & 0, \qquad \forall k = 1, \hdots, N; r = 0, \hdots, N-1.\end{eqnarray}
Interpretation: A forward Hamiltonian cycle is not permitted to travel backwards over a fixed cracker arc.
\begin{eqnarray}x^{u_1}_{r,u_1v_1} & = & \delta^r_0, \qquad \forall r = 0, \hdots, N-1,\\
x^{v_2}_{r,v_2u_2} & = & \delta^r_0, \qquad \forall r = 0, \hdots, N-1.\end{eqnarray}
Interpretation: Starting at the head vertex of a cracker arc, the first step must be to cross the cracker, and then that cracker arc should not be crossed again.
\begin{eqnarray}x^{u_1}_{r,v_1a} & = & \delta^r_1, \qquad \forall r = 0, \hdots, N-1; a \in \mathcal{A}(v_1),\\
x^{v_2}_{r,u_2a} & = & \delta^r_1, \qquad \forall r = 0, \hdots, N-1; a \in \mathcal{A}(u_2).\end{eqnarray}
Interpretation: Starting at the tail vertex of a cracker arc, the head vertex of the cracker arc should only be departed from at stage 1.
\begin{eqnarray}x^{v_1}_{r,u_1v_1} & = & \delta^r_{N-1}, \qquad \forall r = 0, \hdots, N-1,\\
x^{u_2}_{r,v_2u_2} & = & \delta^r_{N-1}, \qquad \forall r = 0, \hdots, N-1.\end{eqnarray}
Interpretation: Starting at the head vertex of a cracker arc, the tail vertex of the cracker arc should only be departed from at stage $N-1$.
\begin{eqnarray}x^{u_1}_{r,v_2u_2} & = & \delta^r_{|V|}, \qquad \forall r = 0, \hdots, N-1,\\
x^{v_2}_{r,u_1v_1} & = & \delta^r_{|U|}, \qquad \forall r = 0, \hdots, N-1,\\
\sum_{r=0}^{N-1} r \left[\sum_{a \in \mathcal{A}(u_1)} x^k_{r,u_1a} - \sum_{a \in \mathcal{A}(v_2)} x^k_{r,v_2a}\right] & = & |U|, \qquad \forall k \in V,\\
\sum_{r=0}^{N-1} r \left[\sum_{a \in \mathcal{A}(u_1)} x^k_{r,u_1a} - \sum_{a \in \mathcal{A}(v_2)} x^k_{r,v_2a}\right] & = & -|V|, \qquad \forall k \in U,\\
x^k_{r,u_1v_1} - x^k_{r+|V|,v_2u_2} & = & 0, \qquad \forall k \in U; r = 0, \hdots, N-1-|V|,\nonumber\\
& & \\
x^k_{r,v_2u_2} - x^k_{r+|U|,u_1v_1} & = & 0, \qquad \forall k \in V; r = 0, \hdots, N-1-|U|.\nonumber\\
& & \end{eqnarray}
Interpretation: The distance between the two head vertices of the two selected cracker arcs should be equal to the number of vertices in the component that must be traversed.
\begin{eqnarray}\sum_{r=0}^{N-1} r \left[\sum_{i \in U}\sum_{a \in \mathcal{A}(i)} x^k_{r,ia} - |U|\sum_{a \in \mathcal{A}(u_2)} x^k_{r,u_2a}\right] & = & \displaystyle\frac{|U|\left(|U|-1\right)}{2}, \qquad \forall k \in V,\\
\sum_{r=0}^{N-1} r \left[\sum_{i \in V}\sum_{a \in \mathcal{A}(i)} x^k_{r,ia} - |V|\sum_{a \in \mathcal{A}(v_1)} x^k_{r,v_1a}\right] & = & \displaystyle\frac{|V|\left(|V|-1\right)}{2}, \qquad \forall k \in U.\end{eqnarray}
Interpretation: Starting from any vertex in one component, a Hamiltonian cycle must first go to the cracker arc that exits the component, then visit every vertex in the other component before reaching the other cracker arc.
\begin{eqnarray}x^k_{r,ia} & = & 0, \qquad \forall k \in U; i \in U; |U| \leq r \leq |V|.\end{eqnarray}
Interpretation: If a Hamiltonian cycle originates in the smaller component, there must be some number of stages in the middle of the cycle when it cannot inhabit the
smaller component in order to visit all the vertices in the larger component.

\end{document}